\documentclass[review]{elsarticle}
\usepackage[usenames]{color}
\usepackage{extarrows}
\usepackage{lineno,hyperref}
\modulolinenumbers[5]

\journal{Journal of \LaTeX\ Templates}

\usepackage{amsfonts}
\usepackage{amsmath, cases}
\usepackage{indentfirst}
\usepackage{graphicx}
\usepackage{latexsym,bm, amsthm}

\usepackage{epsfig}
\usepackage{latexsym}
\usepackage{amssymb}

\begin{document}

\begin{frontmatter}

\title{ Core-EP Decomposition and its Applications }

\author{Hongxing Wang}
\ead{winghongxing0902@163.com}
\address{Department of Financial Mathematics,
Huainan Normal University, Huainan, 232038,   China}


\begin{abstract}
%
%
In this paper,
we introduce the notion of the core-EP  decomposition
 and
some of its properties.
By using the decomposition,
we
derive several  characterizations of the core-EP inverse,
introduce a  pre-order(i.e.the core-EP order)
and
a partial order(i.e. the core-minus partial order),
and
characterize the properties of both orders.
\end{abstract}


\end{frontmatter}

\linenumbers

\section{Introduction}
\numberwithin{equation}{section}  

\newtheorem{theorem}{T{\scriptsize HEOREM}}[section]

\newtheorem{lemma}[theorem]{L{\scriptsize  EMMA}}

\newtheorem{corollary}[theorem]{C{\scriptsize OROLLARY}}
\newtheorem{proposition}[theorem]{P{\scriptsize ROPOSITION}}
\newtheorem{remark}{R{\scriptsize  EMARK}}[section]
\newtheorem{definition}{D{\scriptsize  EFINITION}}[section]

\newtheorem{algorithm}{A{\scriptsize  LGORITHM}}[section]
\newtheorem{example}{E{\scriptsize  XAMPLE}}[section]
\newtheorem{problem}{P{\scriptsize  ROBLEM}}[section]

\newtheorem{assumption}{A{\scriptsize  SSUMPTION}}[section]

In this paper, we use the following notations.
The symbol $ {\mathbb{C}}_{m, n}$ is a set of $m\times n$ matrices with complex entries;
 $A^\ast  $, ${\mathcal{R}}(A)$ and ${\rm rk}\left( A \right)$
 represent
the conjugate transpose, range space  (or column space) and rank
of $A \in {\mathbb{C} }_{m, n} $.
Let $A\in\mathbb{C}_{n, n}$,
  the smallest positive integer $k$,
   which satisfies
   ${\rm rk}\left( A^{k+1} \right)={\rm rk}\left( A^k \right)$,
 is called the index of $A$ and is  denoted by  ${\rm Ind}(A)$.
The index of a non-singular matrix $A$ is $0$ and the index
of a null matrix is $1$.
%
%
%
%
The { Moore-Penrose inverse} of $A\in\mathbb{C}_{m, n}$ is defined as
    the unique matrix $X\in\mathbb{C}_{n, m}$ satisfying  the equations
    \begin{align*}
 (1)~AXA = A, \ \  (2)~XAX = X, \ \
 (3)~\left( {AX} \right)^\ast = AX, \ \
  (4)~\left( {XA} \right)^\ast = XA,
 \end{align*}
and is usually denoted by $X = A^{\dag}$;
the Drazin inverse  of $A \in {\mathbb{C} }_{n,n}$ is defined as
    the unique matrix $X\in\mathbb{C}_{n, n}$ satisfying the equations
\begin{align*}
(1^k)~AXA^k = A^k, \ \ (2)~ XAX = X, \ \ (5)~ AX = XA,
\end{align*}
and is usually denoted by $X = A^D$.
In  particular,
 when  $A\in \mathbb{C}^{\mbox{\rm\footnotesize\texttt{CM}}}_{n }$,
 the matrix  $X$   is called the group inverse of $A$,
 and is denoted by $X = A^\#$({see \rm\cite{Ben2003book}}),
 in which
 $$\mathbb{C}^{\mbox{\rm\footnotesize\texttt{CM}}}_{n }
 =\left\{A\left|A\in\mathbb{C}_{n, n},
 \mbox{\rm rk}\left(A^2\right)= \mbox{\rm rk}\left(A\right) \right.\right\}.$$
%
%
The core inverse of $A \in {\mathbb{C} }_{n}^{\mbox{\rm\footnotesize\texttt{CM}}}$
is defined as the unique matrix $X\in\mathbb{C}_{n,n}$ satisfying
\begin{align}
\label{1-1}
AX = AA^\dag,~~  {\mathcal{R}}\left( {X} \right) \subseteq  {\mathcal{R}}\left( A\right)
\end{align}
and is denoted by $X = A^{\tiny\textcircled{\#}}$, {\rm\cite{Baksalary2010lma}}.
When $A\in\mathbb{C}^{\mbox{\rm\footnotesize\texttt{CM}}}_{n }$,
we call it a  core-invertible (or group-invertible) matrix.
%
%
%

\bigskip

%
%
According to   \cite[Corollary 6]{Hartwig1984lma}, 
for every matrix $A\in\mathbb{C}_{n,n}$ of rank $r$,
there exists  a   unitary matrix $U$
 such that 
\begin{align}
\label{1-2}
A = U\left[ {{\begin{matrix}
 T   & S   \\
 0   & 0   \\
\end{matrix} }} \right]U^\ast,
\end{align}
where  $T\in\mathbb{C}_{r, r}$ and $S\in\mathbb{C}_{r, n-r}$.
 When $A\in\mathbb{C}^{\mbox{\rm\footnotesize\texttt{CM}}}_{n }$
 and
 $A$ is of the form (\ref{1-2}),
 then
 $T$ is non-singular,
 and (\ref{1-2}) is called a core decomposition in this paper.
Meanwhile,
  the core inverse of $A$ is of the form (\ref{1-3}).
 \begin{lemma}
{\rm\cite{Baksalary2010lma}}
 Let $A \in { \mathbb{C}}_n^{\mbox{\rm\footnotesize\texttt{CM}}}$ .
 If $A$ is of the form (\ref{1-2}),
 then
\begin{align}
\label{1-3}
A^{\tiny\textcircled{\#}} = U\left[ {{\begin{matrix}
 {T^{ - 1}}   & 0   \\
 0   & 0   \\
\end{matrix} }} \right]U^\ast ,
\end{align}
where $U$ is a   unitary matrix,
 $T$ is a non-singular matrix,
and
\begin{align}
\label{1-3-1}
AA^{\tiny\textcircled{\#}}  = AA^\dag .
\end{align}
\end{lemma}


By applying the decomposition in \cite[Corollary 6]{Hartwig1984lma},
some mathematical researchers
delve into
 characterizations and properties of
 the  core inverse,
the core partial order
and some relevant issues
in \cite[etc]{Baksalary2010lma,Baksalary2014amc,Manjunatha2014lma,Wang2015lma}.
Note that
 not every matrix  has \textcolor[rgb]{1.00,0.00,0.00}{ the } core inverse,
and
the notion of the core inverse extends to any square matrix 
whose index is not limited to less than or equal to one,
\cite{Baksalary2014amc,Malik2014amc,Manjunatha2014lma}.

Based on the decomposition in \cite[Corollary 6]{Hartwig1984lma},
Ben\'{\i}tez
introduces  a new decomposition for square matrices
in Theorem 2.1 of \cite{Benitez2010ela}.
Ben\'{\i}tez and Liu give the new decomposition  a new proof by CS decomposition
 in \cite{Benitez2013laa}.
 The  decomposition \cite[Theorem 2.1]{Benitez2010ela}
is used to study partial orders, general inverses, etc,
\cite{Benitez2010ela,Benitez2013laa}.

In this paper,
we look at the decomposition\cite[Corollary 6]{Hartwig1984lma} in another way.
Based on the core decomposition,
we introduce
the  core-EP decomposition
for a square matrix over the complex field
in Section \ref{Section-core-EP-D}.
By applying the decomposition,
we derive
some new characterizations of the core-EP inverse
in Section \ref{Section-core-EP-Inverse},
and
characterize the properties
of a pre-order:  the core-EP order
in Section \ref{Section-core-EP-Order},
and
of a partial order: the core-minus partial order
in Section \ref{Section-Core-Minus-Order}, respectively.

\section{The core-EP decomposition}
\label{Section-core-EP-D}

It is well known that the core-nilpotent decomposition
have been widely used in matrix theory \cite{Mitra2010book}:
  Let $A\in \mathbb{C}_{n,n}$ with ${\rm Ind}(A)=k$.
Then $A$ can be written as the sum of matrices $A_1$ and $A_2$ i.e.
$A = A_1 + A_2$ where
$$A_1\in { \mathbb{C}}_n^{\mbox{\rm\footnotesize\texttt{CM}}},  \
A^k_2 = 0
 \mbox{\ \rm and\ \ }
 A_1  A_2 = A_2A_1 = 0.$$
Similarly,
we introduce the notion of the core-EP of decomposition in Theorem \ref{core-EP-Decomposition}.

\begin{theorem}
[Core-EP Decomposition]
\label{core-EP-Decomposition}

Let $A\in \mathbb{C}_{n,n}$ with ${\rm Ind}(A)=k$.
Then $A$ can be written as the sum of matrices $A_1$ and $A_2$ i.e.
$A = A_1 + A_2$ where

{\rm (i)}\ \   $A_1\in { \mathbb{C}}_n^{\mbox{\rm\footnotesize\texttt{CM}}}$;

{\rm (ii)}\ \     $A^k_2 = 0 $;

{\rm (iii)}\ \   $A_1^\ast A_2 = A_2A_1 = 0$.

\noindent Here one or both of $A_1$ and $A_2$ can be null.
\end{theorem}

\begin{proof}
When $A$ is a non-singular matrix, i.e. ${\rm Ind}(A)=0$,
it is easy to check that
$A_1 = A$ and $A_2 = 0$ satisfy
(i), (ii) and (iii).
When $A=0$, i.e. ${\rm Ind}(A)=1$,
it is easy to check that
$A_1 = 0$ and $A_2 = 0$ satisfy
(i), (ii) and (iii).
When $A$ is nilpotent with ${\rm Ind}(A)=k$,
it is easy to check that
$A_1 = 0$ and $A_2 = A$ satisfy
(i), (ii) and (iii).
So, the matrix $A$
which we consider in the rest of the proof
is a square, singular, non-null and non-nilpotent matrix.

If  $\mbox{Ind}\left( A \right) = k> 1$,
then
$A^k\in { \mathbb{C}}_n^{\mbox{\rm\footnotesize\texttt{CM}}}$.
 Write
 \begin{equation}
 \label{2-1}
\begin{aligned}
A_1
&
=
A^k\left( {A^k} \right)^{\tiny\textcircled{\#}} A,
\\
A_2
&
=
A - A^k\left( {A^k}\right)^{\tiny\textcircled{\#}} A. 
\end{aligned}
\end{equation}
Since
$
{\rm rk}\left(   A^{k+1} \right) =
{\rm rk}\left(   A^{k} \right) $
and
$
{\rm rk}\left(   A_1 \right)
\leq
{\rm rk}\left(   A^k \right)$,
it follows that
\begin{align*}
{\rm rk}\left(   A^{k+1} \right)
&
 =
{\rm rk}\left( A A^k\left( {A^k} \right)^{\tiny\textcircled{\#}} A^k \right)
 \leq
{\rm rk}\left( A^k A\left( {A^k} \right)^{\tiny\textcircled{\#}} A \right)
  \\
  &
  =
  {\rm rk}\left( A^k\left( {A^k} \right)^ {\tiny\textcircled{\#}} A^kA\left( {A^k} \right)^{\tiny\textcircled{\#}} A\right)
  \\
&
  =
{\rm rk}\left( A_1^2\right)
\leq
{\rm rk}\left(   A_1 \right)
\leq
{\rm rk}\left(   A^k \right).
\end{align*}
Therefore,
 we have
 $A_1\in { \mathbb{C}}_n^{\mbox{\rm\footnotesize\texttt{CM}}}$.
By applying (\ref{1-3-1}),
we derive
\begin{align*}
 A_2 A_1
 &
  =
  \left( {A - A^k\left( {A^k} \right)^{\tiny\textcircled{\#}} A} \right)A^k\left(
{A^k} \right)^{\tiny\textcircled{\#}} A \\
&
= AA^k\left( {A^k} \right)^{\tiny\textcircled{\#}} A - A^k\left( {A^k}
\right)^{\tiny\textcircled{\#}} AA^k\left( {A^k} \right)^{\tiny\textcircled{\#}} A = 0,
\\
 A_1^\ast  A_2
 &
  =
  \left( {A^k\left( {A^k} \right)^{\tiny\textcircled{\#}} A} \right)^\ast  \left( {A
- A^k\left( {A^k} \right)^{\tiny\textcircled{\#}} A} \right)
\\
&
 =
 A^\ast  \left\{ {A^k\left( {A^k}
\right)^{\tiny\textcircled{\#}} \left( {A - A^k\left( {A^k} \right)^{\tiny\textcircled{\#}} A} \right)}
\right\} = 0 .
\end{align*}
By using $A_2 A_1= 0$,
it is easy to check that
\begin{align*}
 A_2^k
&
 =
  A_2 \left( {A_1 + A_2 }
\right)^{k - 1} = A_2 A^{k - 1}
\\
&
=
\left( {A - A^k\left( {A^k} \right)^{\tiny\textcircled{\#}}
A} \right)A^{k - 1} = 0.
\end{align*}
\end{proof}

Since the case of $A$ being nilpotent(or null, or no-singular) is considered to be trivial,
  we consider the matrix $A$ which is a square, singular, non-null and non-nilpotent
matrix in the rest of the paper, unless indicated otherwise .


\begin{theorem}
\label{Theorem-3-6}
 The core-EP decomposition of a given matrix  is unique.
\end{theorem}

\begin{proof}
 Let   $A = A_1 + A_2 $ be the core-EP decomposition of $A\in\mathbb{C}_{n,n}$,
 where  $ A_1$ is a core-invertible matrix,  and $A_2$ is a nilpotent matrix.
And let $A_1$ and $A_2$ be as in (\ref{2-1}).

Let
$A = B_1 + B_2 $ be another core-EP decomposition of $A$,
where $B_1$ and $B_2$ satisfy
the conditions (i), (ii) and (iii) in Theorem \ref{core-EP-Decomposition}.
Then
$
A^k={\sum\limits_{i = 0}^k {B_1^iB_2^{k - i}} }$.
By using
$B_1^\ast B_2=0$
and
$B_2^k=0$,
we obtain
$\left(A^k\right)^\ast B_2=0$.
Therefore,
$\left(A^k\right)^\dag B_2=0$.
By using
$B_2B_1=0$
and
$B_1\in { \mathbb{C}}_n^{\mbox{\rm\footnotesize\texttt{CM}}}$,
Then
$A^k B_1 \left(B_1^k\right)^\#=B_1$.
Since
$A^k\left( {A^k} \right)^{\dag}B_1
=A^k\left( {A^k} \right)^{\dag}A^k B_1 \left(B_1^k\right)^\#
=A^k B_1 \left(B_1^k\right)^\#
=B_1^{k+1}\left(B_1^k\right)^\#
B_1$.
It follows that
\begin{align}
\nonumber
B_1-A_1
&
=
B_1-A^k\left( {A^k} \right)^{\tiny\textcircled{\#}} A
=
B_1-A^k\left( {A^k} \right)^{\dag} A
\\
\nonumber
&
=
B_1-A^k\left( {A^k} \right)^{\dag}B_1-A^k\left( {A^k} \right)^{\dag}B_2
\\
\nonumber
&
=
0,
\end{align}
that is,
$B_1=A_1$.
Therefore,
the core-EP decomposition of $A$  is unique.
\end{proof}

\begin{theorem}
\label{Theorem-2-2}
 Let $A = A_1 + A_2 $ is the core-EP decomposition  of $A$.
 Then there exists a  unitary matrix $U$
 such that
\begin{align}
\label{2-2}
A_1 = U\left[ {{\begin{matrix}
 T   & S   \\
 0   & 0   \\
\end{matrix} }} \right]U^\ast
\
{\rm and}
\
A_2 = U\left[ {{\begin{matrix}
 0   & 0   \\
 0   & N   \\
\end{matrix} }} \right]U^\ast  ,
\end{align}
where $T$ is   non-singular, and $N$ is  nilpotent.
\end{theorem}

\begin{proof}
For  $A_1\in { \mathbb{C}}_n^{\mbox{\rm\footnotesize\texttt{CM}}}$,
let
\begin{align*}
A_1 = U\left[ {{\begin{matrix}
 T   & S   \\
 0   & 0   \\
\end{matrix} }} \right]U^\ast,
\end{align*}
where $T$ is   non-singular, and $U$ is unitary.
Write
\begin{align*}
A_2 = U\left[ {{\begin{matrix}
 {X_1 }   & {X_2 }   \\
 {X_3 }   & N   \\
\end{matrix} }} \right]U^\ast  .
\end{align*}
By applying $A_1^\ast  A_2 = 0$,
we have
\begin{align*}
A_1^\ast  A_2 = U\left[ {{\begin{matrix}
 {T^\ast  }   & 0   \\
 {S^\ast  }   & 0   \\
\end{matrix} }} \right]\left[ {{\begin{matrix}
 {X_1 }   & {X_2 }   \\
 {X_3 }   & N   \\
\end{matrix} }} \right]U^\ast  = U\left[ {{\begin{matrix}
 {T^\ast  X_1 }   & {T^\ast  X_2 }   \\
 {S^\ast  X_1 }   & {S^\ast  X_2 }   \\
\end{matrix} }} \right]U^\ast  = 0.
\end{align*}
Therefore,
$T^\ast  X_1 = 0$ and $T^\ast  X_2 = 0$.
Since $T$ is invertible,
we obtain  $X_1 = 0$ and $X_2 = 0$.
By applying $A_2 A_1 = 0$,
we have
\begin{align*}
A_2 A_1 = U\left[ {{\begin{matrix}
 0   & 0   \\
 {X_3 }   & N   \\
\end{matrix} }} \right]\left[ {{\begin{matrix}
 T   & S   \\
 0   & 0   \\
\end{matrix} }} \right]U^\ast  = U\left[ {{\begin{matrix}
 0   & 0   \\
 {X_3 T}   & {X_3 S}   \\
\end{matrix} }} \right]U^\ast ,
\end{align*}
Therefore, $X_3 = 0$
 and
\begin{align*}
A_2 = U\left[{{\begin{matrix}
 0   & 0   \\
 0   & N   \\
\end{matrix} }} \right]U^\ast.
\end{align*}%
Since $A_2 $ is  nilpotent,  it is clear that $N$ is  nilpotent.
\end{proof}

\bigskip

In the following sections,
we consider
the core-EP inverse,
the core-EP order
and  the core-minus partial order
by using the core-EP decomposition.

\section{Characterizations of core-EP inverse}
\label{Section-core-EP-Inverse}

When the index of a given square matrix is less than or equal to one,
by applying the decomposition in \cite[Corollary 6]{Hartwig1984lma},
mathematical researchers introduce  the core inverse (\ref{1-1}) and (\ref{1-2}).
In\cite{Baksalary2014amc,Malik2014amc,Manjunatha2014lma},
the notion of the core inverse extends to any square matrix
whose index is not limited to less than or equal to one.
For example,
let $A\in { \mathbb{C}}_{n,n}$ be of the form (\ref{1-2}),
then 
 the generalized core inverse
 $A^{\tiny\textcircled{\dag}}$
 can be expressed as
 \cite[Theorem 3.5 and Remark 2]{Manjunatha2014lma}:
 \begin{align}
\label{1-5}
A^{\tiny\textcircled{\dag}}
= A^k\left( {\left( {A^\ast  } \right)^kA^{k + 1}} \right)^ -A^k,
\end{align}
where $k={\rm Ind}(A)$.
We also call
 the generalized core inverse
 $A^{\tiny\textcircled{\dag}}$
the core-EP inverse of $A$.
%
%
When $k\leq 1$,
it is easy to verify that
 the   core-EP inverse  coincide with
 the  core inverse\cite[Definition 1]{Baksalary2010lma}.
Some properties and characterizations of the core inverse，
the core-EP inverse and
other generalized core inverses
 can be seen in \cite{Baksalary2014amc,Malik2014amc,Manjunatha2014lma,Wang2015lma}
\begin{lemma}{\rm\cite[Lemma 3.3]{Manjunatha2014lma}}
\label{Lemma-2}
 Let  $A\in \mathbb{C}_{n, n}$ with
${\rm Ind}\left( A \right) = k$.
Then
  $G$ is a core--EP inverse of $A$
if and only if
  $G$ is a matrix satisfying the conditions:
\begin{align*}
\left( {1^k} \right)~ XA^{k + 1} = A^k, \ \
\left( 2 \right)~ XAX = X,  \ \
\left(3 \right)~ \left( {AX} \right)^\ast  = AX ,
\end{align*}%
and
$\mathcal{R}\left( G \right) \subseteq \mathcal{R}\left( {A^k}\right)$.
\end{lemma}

%


In this section,
we give some characterizations of the core-EP inverse by using the core-EP decomposition.
\begin{theorem}
\label{Theorem-4-0}
 Let   $A = A_1 + A_2 $ be the core-EP decomposition of $A\in\mathbb{C}_{n,n}$,
 where  $ A_1$ is a core-invertible matrix,  and $A_2$ is a nilpotent matrix.
 Then
 $$A^{\tiny\textcircled{\dag}}= A_1^{\tiny\textcircled{\#}}.$$
\end{theorem}

\begin{proof}
Let $A_1$ be    as in (\ref{2-2}),
then
$$A_1^{\tiny\textcircled{\#}} = U\left[ {{\begin{matrix}
 {T^{ - 1}}   & 0   \\
 0   & 0   \\
\end{matrix} }} \right]U^\ast.
$$
Write
$$A^k = U\left[ {{\begin{matrix}
 {T^k}   & \widehat{T}   \\
 0   & 0   \\
\end{matrix} }} \right]U^\ast.$$
where $\widehat{T} $ is a corresponding matrix.
It is easy to see  that
$\mathcal{R}\left( A_1^{\tiny\textcircled{\#}} \right) \subseteq \mathcal{R}\left( {A^k} \right)$
and
\begin{align*}
 &
 A_1^{\tiny\textcircled{\#}} A^{k + 1}
 =
  U\left[{{\begin{matrix}
 {T^{ - 1}}   & 0   \\
 0   & 0   \\
\end{matrix} }} \right] \left[ {{\begin{matrix}
 {T^{k + 1}}   & {T\widehat{T}}   \\
 0   & 0   \\
\end{matrix} }} \right]U^\ast  = U\left[{{\begin{matrix}
 {T^k}   & {\widehat{T}}   \\
 0   & 0   \\
\end{matrix} }} \right]U^\ast  = A^k,
\\
 &
 A_1^{\tiny\textcircled{\#}} AA_1^{\tiny\textcircled{\#}}
 =
  U\left[ {{\begin{matrix}
 {T^{ - 1}}   & 0   \\
 0   & 0   \\
\end{matrix} }} \right]\left[{{\begin{matrix}
 T   & S   \\
 0   & N   \\
\end{matrix} }} \right]\left[{{\begin{matrix}
 {T^{ - 1}}   & 0   \\
 0   & 0   \\
\end{matrix} }} \right]U^\ast  = U\left[ {{\begin{matrix}
 {T^{ - 1}}   & 0   \\
 0   & 0   \\
\end{matrix} }} \right]U^\ast  = A_1^{\tiny\textcircled{\#}} ,
\\
 &
 \left( {AA_1^{\tiny\textcircled{\#}} } \right)^\ast
 =
 \left[ {U\left[
{{\begin{matrix}
 T   & S   \\
 0   & N   \\
\end{matrix} }} \right]\left[ {{\begin{matrix}
 {T^{ - 1}}   & 0   \\
 0   & 0   \\
\end{matrix} }} \right]U^\ast  } \right]^\ast  = \left[ {U\left[
{{\begin{matrix}
 I   & S   \\
 0   & N   \\
\end{matrix} }} \right]U^\ast  } \right]^\ast  = AA_1^{\tiny\textcircled{\#}} .
\end{align*}
Therefore,
by applying Theorem \ref{Theorem-2-2},
we have
$A^{\tiny\textcircled{\dag}}= A_1^{\tiny\textcircled{\#}} $.
\end{proof}

\begin{corollary}
\label{Theorem-4-0-1}
 Let    $A = A_1 + A_2 $ be
 the core-EP decomposition of $A\in\mathbb{C}_{n,n}$,
 where  $ A_1$ is a core-invertible matrix,  and $A_2$ is a nilpotent matrix.
 Let the decompositions of $ A_1$ and $ A_2$ be as in (\ref{2-2}).
 Then
 \begin{align}
 \label{Th-4-0-1-1}
 A^{\tiny\textcircled{\dag}}
 =
 U\left[ {{\begin{matrix}
 {T^{ - 1}}   & 0   \\
 0   & 0   \\
\end{matrix} }} \right]U^\ast.
\end{align}
\end{corollary}

\begin{corollary}
\label{Cor-3-1}
Let   $A = A_1 + A_2 $ be the core-EP decomposition of $A\in\mathbb{C}_{n,n}$,
 where  $ A_1$ is a core-invertible matrix,  and $A_2$ is a nilpotent matrix.
 Then \begin{align}
 A^{\tiny\textcircled{\dag}}
 &
 =
  A^k\left( {A^{k + 1}} \right)^{\tiny\textcircled{\#}},
  \\
\label{3-2}
 AA^{\tiny\textcircled{\dag}}
 &
  =
   A^k\left( {A^k} \right)^{\tiny\textcircled{\#}} = A^k\left( {A^k} \right)^\dag,
   \\
\label{3-3}
   A_1
   &
   =AA^{\tiny\textcircled{\dag}} A,
   \\
\label{3-4}
   A_2
   &
   =A-AA^{\tiny\textcircled{\dag}} A.
 \end{align}
\end{corollary}


\section{The core-EP order}
\label{Section-core-EP-Order}

%
%
%

A binary
relation on a non-empty set is said to be a pre-order if it is reflexive and
transitive. If it is also anti-symmetric, then it is called a partial order
\cite[Charp 1]{Mitra2010book}.
 As is noted in \cite[Section 3]{Baksalary2010lma}
the core partial order is given:
$$A \mathop \le \limits^ {\tiny\textcircled{\#}} B
:
A,B \in {\mathbb{C} }_{n}^{\mbox{\rm\footnotesize\texttt{CM}}},
 A^ {\tiny\textcircled{\#}}A = A^ {\tiny\textcircled{\#}}B
{\ \rm and\ }
AA^ {\tiny\textcircled{\#}} =BA^ {\tiny\textcircled{\#}}.$$
Some characterizations of the core partial order are given in \cite{Baksalary2010lma,Wang2015lma}.
\begin{lemma}
{\rm\cite{Baksalary2010lma}}
\label{lemma-Bak-2}
Let $A,B \in {\mathbb{C} }_{n}^{\mbox{\rm\footnotesize\texttt{CM}}} $,
and let $A$ be of the form {\rm(\ref{1-2})}.
Then the following conditions are equivalent:
 \begin{enumerate}
   \item[{\rm(i)}]  $A\mathop \le \limits^ {\tiny\textcircled{\#}} B$;
   \item[{\rm(ii)}]  $B = U\left[ {{\begin{matrix}
 T    & S   \\
 0    & Z    \\
\end{matrix} }} \right]U^\ast  $,
where $T$ is nonsingular and $Z \in {\mathbb{C} }_{n - r,n - r} $ is some matrix of index one;
   \item[{\rm(iii)}]  $A^\dag A = A^\dag B$, $A^2 = BA$.
 \end{enumerate}
\end{lemma}

It is of interest to consider
  the  binary operation:
\begin{align}
\label{4-1}
A \mathop \le \limits^{\tiny\textcircled{\dag}} B
:
A,B \in {\mathbb{C} }_{n,n},
 A^{\tiny\textcircled{\dag}} A = A^{\tiny\textcircled{\dag}} B
 {\ \rm and\ }
AA^{\tiny\textcircled{\dag}}=BA^{\tiny\textcircled{\dag}}.
\end{align}
We call it the core-EP order.

\begin{remark}
It is noteworthy that
 $A^{\tiny\textcircled{\dag}}=A^{\tiny\textcircled{\#}}$
 and
$ B^{\tiny\textcircled{\dag}}=B^{\tiny\textcircled{\#}}$,
 when $A$ and $B
 \in {\mathbb{C} }_{n}^{\mbox{\rm\footnotesize\texttt{CM}}}$.
Therefore,
  the core-EP order and  the core partial order coincide
  in ${\mathbb{C} }_{n}^{\mbox{\rm\footnotesize\texttt{CM}}}$.
  \end{remark}

In the following theorem,
we derive some characterizations of the  binary operation.

\begin{theorem}
\label{Theorem-4-1}
Let $A$ and $B$ be square matrices of the same order over the complex field.
Then the following are equivalent:

\begin{itemize}
  \item[{\rm(i)}]
   $A\mathop \le \limits^{\tiny\textcircled{\dag}} B:A^{\tiny\textcircled{\dag}}A
= A^{\tiny\textcircled{\dag}} B,
\
AA^{\tiny\textcircled{\dag}} = BA^{\tiny\textcircled{\dag}}$;

  \item[{\rm(ii)}]
  There exists a  unitary matrix $U$ such that
\begin{align}
\label{4-2}
A = U\left[ {{\begin{matrix}
 {T_1 }   & {T_2 }   & {S_1 }   \\
 0   & {N_{11} }   & {N_{12} }   \\
 0   & {N_{13} }   & {N_{14} }   \\
\end{matrix} }} \right]U^\ast  ,B = U\left[ {{\begin{matrix}
 {T_1 }   & {T_2 }   & {S_1 }   \\
 0   & {T_3 }   & {S_2 }   \\
 0   & 0   & {N_2 }   \\
\end{matrix} }} \right]U^\ast ,
 \end{align}
where
$\left[ {{\begin{matrix}
 {N_{11} }   & {N_{12} }   \\
 {N_{13} }   & {N_{14} }   \\
\end{matrix} }} \right]$ and $N_2 $ are nilpotent,  $T_1$ and $T_3 $ are  non-singular;

  \item[{\rm(iii)}]
  $A^{k + 1} = BA^k$, $A^\ast  A^k = B^\ast  A^k$, where $k$ is the index of $A$;

  \item[{\rm(iv)}]
$VA V^\ast\mathop \le \limits^{\tiny\textcircled{\dag}} VBV^\ast$
for each unitary matrix $V$;

  \item[{\rm(v)}]
  $A_1 \mathop \le \limits^ {\tiny\textcircled{\#}} B_1$,
 where
$ A_1$  and  $ B_1$ are  core-invertible,
$A_2$ and  $ B_2$ are   nilpotent,
 and  $A = A_1 + A_2 $ and
 $B= B_1 + B_2 $ be the core-EP decompositions of $A$ and $B$,respectively.
\end{itemize}
\end{theorem}

\begin{proof}
(i) $\Rightarrow$ (ii). \ \
Let
\begin{align*}
A = U^\ast  \left[ {{\begin{matrix}
 {T_1 }   & {\widehat{T}_2 }   & {\widehat{S}_1 }   \\
 0   & {\widehat{N}_{11} }   & {\widehat{N}_{12} }   \\
 0   & {\widehat{N}_{13} }   & {\widehat{N}_{14} }   \\
\end{matrix} }} \right]U_1^\ast  ,
 \end{align*}
be a  core-EP decomposition of $A$,
where $T_1 $ is non-singular,
$\left[ {{\begin{matrix}
 {\widehat{N}_{11} }   & {\widehat{N}_{12} }   \\
 {\widehat{N}_{13} }   & {\widehat{N}_{14} }   \\
\end{matrix} }} \right]$ is  nilpotent,
and  $U$ is unitary.
Then
\begin{align*}
A^{\tiny\textcircled{\dag}} = U_1 \left[ {{\begin{matrix}
 {T_1^{ - 1} }   & 0   & 0   \\
 0   & 0   & 0   \\
 0   & 0   & 0   \\
\end{matrix} }} \right]U_1^\ast  .
 \end{align*}
Write
\begin{align*}
B = U_1 \left[ {{\begin{matrix}
 {X_{11} }   & {X_{12} }   & {X_{13} }   \\
 {X_{21} }   & {X_{22} }   & {X_{23} }   \\
 {X_{31} }   & {X_{32} }   & {X_{33} }   \\
\end{matrix} }} \right]U_1^\ast  .
 \end{align*}
Since,
\begin{align*}
 A^{\tiny\textcircled{\dag}}A
 &
  = U_1 \left[{{\begin{matrix}
 I   & {T_1^{ - 1} \widehat{T}_2 }   & {\widehat{T}_1^{ - 1}
\widehat{S}_1 }   \\
 0   & 0   & 0   \\
 0   & 0   & 0   \\
\end{matrix} }} \right]U_1^\ast  = A^{\tiny\textcircled{\dag}}B = U_1 \left[
{{\begin{matrix}
 {T_1^{ - 1} X_{11} }   & {T_1^{ - 1} X_{12} }   & {T_1^{ - 1}
X_{13} }   \\
 0   & 0   & 0   \\
 0   & 0   & 0   \\
\end{matrix} }} \right]U_1^\ast  \\
 &
\Rightarrow X_{11} = T_1 ,X_{12} = \widehat{T}_2 ,X_{13} = \widehat{S}_1 ;
\\
 AA^{\tiny\textcircled{\dag}}
 &
 = U_1 \left[{{\begin{matrix}
 I   & 0   & 0   \\
 0   & 0   & 0   \\
 0   & 0   & 0   \\
\end{matrix} }} \right]U_1^\ast  = BA^{\tiny\textcircled{\dag}} = U_1 \left[
{{\begin{matrix}
 I   & 0   & 0   \\
 {X_{21} T_1^{ - 1} }   & 0   & 0   \\
 {X_{31} T_1^{ - 1} }   & 0   & 0   \\
\end{matrix} }} \right]U_1^\ast  \\
 &
 \Rightarrow X_3 = 0,X_4 = 0,
 \end{align*}
we have
\[
B = U_1 \left[ {{\begin{matrix}
 {T_1 }   & {\widehat{T}_2 }   & {\widehat{S}_1 }   \\
 0   & {X_{22} }   & {X_{23} }   \\
 0   & {X_{32} }   & {X_{33} }   \\
\end{matrix} }} \right]U_1^\ast  .
\]

Let
\begin{align*}
\left[{{\begin{matrix}
 {X_{22} }   & {X_{23} }   \\
 {X_{32} }   & {X_{33} }   \\
\end{matrix} }} \right]= U_2 \left[{{\begin{matrix}
 {T_2 }   & {S_2 }   \\
 0   & {N_2 }   \\
\end{matrix} }} \right]U_2^\ast  ,
\end{align*}
be a  core-EP decomposition of $\left[ {{\begin{smallmatrix}
 {X_{22} }   & {X_{23} }   \\
 {X_{32} }   & {X_{33} }   \\
\end{smallmatrix} }} \right]$,
where $T_2 $ is non-singular,
$N_2$ is  nilpotent,
and  $U_2$ is unitary.

Denote
\begin{align*}U = U_1 \left[ {{\begin{matrix}
 I   & 0   \\
 0   & {U_2 }   \\
\end{matrix} }} \right], \
\
\left[{{\begin{matrix}
 {N_{11} }   & {N_{12} }   \\
 {N_{13} }   & {N_{14} }   \\
\end{matrix} }} \right]
=
U_2^\ast  \left[ {{\begin{matrix}
 {\widehat{N}_{11} }   & {\widehat{N}_{12} }   \\
 {\widehat{N}_{13} }   & {\widehat{N}_{14} }   \\
\end{matrix} }} \right]U_2   \end{align*}
 {and  }
 $ \left[ {{\begin{matrix}
 {\widehat{T}_2 }   & {\widehat{S}_1 }   \\
\end{matrix} }} \right]U_2 = \left[ {{\begin{matrix}
 {T_2 }   & {S_1 }   \\
\end{matrix} }} \right]$,
we have
\begin{align*}
A = U\left[ {{\begin{matrix}
 {T_1 }   & {T_2 }   & {S_1 }   \\
 0   & {N_{11} }   & {N_{12} }   \\
 0   & {N_{13} }   & {N_{14} }   \\
\end{matrix} }} \right]U^\ast  ,\quad
B = U\left[{{\begin{matrix}
 {T_1 }   & {T_2 }   & {S_1 }   \\
 0   & {T_3 }   & {S_2 }   \\
 0   & 0   & {N_2 }   \\
\end{matrix} }} \right]U^\ast,
\end{align*}
and
$\left[{{\begin{matrix}
 {N_{11} }   & {N_{12} }   \\
 {N_{13} }   & {N_{14} }   \\
\end{matrix} }} \right]$
is nilpotent.

(ii) $\Rightarrow$ (i) is easy.

(ii) $\Rightarrow$ (iii). \ \
By applying (\ref{4-2}),
we have
 $A^k = U\left[{{\begin{matrix}
 {T_1^k }   & {\overline T _2 }   & {\overline S _1 }   \\
 0   & 0   & 0   \\
 0   & 0   & 0   \\
\end{matrix} }} \right]U^\ast$,
where
${\overline T _2 }  $
and
${\overline S _1 }$
are both corresponding matrices.
It is easy to obtain
\begin{align*}
A^{k + 1}
 &
 = U\left[ {{\begin{matrix}
 {T_1 T_1^k }   & {T_1 \overline T _2 }   & {T_1 \overline S _1 }
  \\
 0   & 0   & 0   \\
 0   & 0   & 0   \\
\end{matrix} }} \right]U^\ast  = BA^k,
\\
A^\ast  A^k
 &
  = U\left[ {{\begin{matrix}
 {T_1^\ast  T_1^k }   & {T_1^\ast  \overline T _2 }   & {T_1^\ast
\overline S _1 }   \\
 {T_2^\ast  T_1^k }   & {T_2^\ast  \overline T _2 }   & {T_2^\ast
\overline S _1 }   \\
 {S_1^\ast  T_1^k }   & {S_1^\ast  \overline T _2 }   & {S_1^\ast
\overline S _1 }   \\
\end{matrix} }} \right]U^\ast  = B^\ast  A^k .
\end{align*}

(iii) $\Rightarrow$ (i). \ \
By applying (\ref{3-2}),
we obtain  
\begin{align*}
 AA^k
 =
 BA^k
 &
 \Rightarrow AA^k\left( {A^{k + 1}} \right)^{\tiny\textcircled{\#}} = BA^k\left(
{A^{k + 1}} \right)^{\tiny\textcircled{\#}}
\\
&
\Rightarrow AA^{\tiny\textcircled{\dag}} = BA^{\tiny\textcircled{\dag}},
\\
 A^\ast  A^k
 = B^\ast  A^k
 &
 \Rightarrow A^\ast  A^k\left( {A^{k + 1}} \right)^{\tiny\textcircled{\#}} =
B^\ast  A^k\left( {A^{k + 1}} \right)^{\tiny\textcircled{\#}} \Rightarrow A^\ast  A^{\tiny\textcircled{\dag}}
= B^\ast  A^{\tiny\textcircled{\dag}} 
 \\
 &
 \Rightarrow A^{\tiny\textcircled{\dag}}
 \left( {\left( {A^{\tiny\textcircled{\dag}}} \right)^\ast  } \right)^{\tiny\textcircled{\#}} \left( {A^{\tiny\textcircled{\dag}}} \right)^\ast  A =
A^{\tiny\textcircled{\dag}}\left( {\left( {A^{\tiny\textcircled{\dag}}} \right)^\ast  }
\right)^{\tiny\textcircled{\#}} \left( {A^{\tiny\textcircled{\dag}}} \right)^\ast  B
\\
 &
 \Rightarrow A^{\tiny\textcircled{\dag}}A = A^{\tiny\textcircled{\dag}}B.
 \end{align*}

(iii) $\Leftrightarrow$ (iv) is easy.

 (ii) $\Leftrightarrow$ (v).\
 Let   $A = A_1 + A_2 $
 and
 $B= B_1 + B_2 $ be the core-EP decompositions
of $A$ and $B\in\mathbb{C}_{n,n}$, respectively,
 where  $ A_1$  and  $ B_1$ are  core-invertible,
 and $A_2$ and  $ B_2$ are   nilpotent.
Then
 $A^{\tiny\textcircled{\dag}}= A_1^{\tiny\textcircled{\#}}  $
 and
 $B^{\tiny\textcircled{\dag}}=B_1^{\tiny\textcircled{\#}}  $.
By applying Lemma \ref{lemma-Bak-2},
we have (ii) $\Leftrightarrow$ (v).
\end{proof}

  By applying Theorem \ref{Theorem-4-1},
we see that
the  binary operation  is reflexive and transitive,
that is,
  the core-EP order
is a pre-order.
But the core-EP order is not  antisymmetric.

\begin{example}
\label{Ex-1}
Let
\begin{align*}
A =  \left[ {{\begin{matrix}
 1   & 2   & 3 \\
 0   & 0   & 0   \\
 0   & 0   & 0  \\
\end{matrix} }} \right] ,\quad
B =\left[{{\begin{matrix}
 1   & 2   & 3  \\
 0   & 0   & 1 \\
 0   & 0   & 0 \\
\end{matrix} }} \right] ,
\end{align*}
Then
$A\mathop \le \limits^{\tiny\textcircled{\dag}} B $
and
$B\mathop \le \limits^{\tiny\textcircled{\dag}} A$.
However, $A \neq B$.
\end{example}

\begin{theorem}
\label{Theorem-5-4}
The core-EP order  is  only a pre-order and
not a partial order.
 \end{theorem}

It is well known that
the Drazin order is a pre-order:
$A\mathop \le \limits^{D} B:A^{k}B =B A^{k} = A^{k+1}$,
where $k$ is the index of $A$.
In the following example,
we see that the Drazin order 
and
the core-EP order are not equivalent.

\begin{example}
\label{Ex-1-1}
Let  $A$ and $B$ be as in Example (\ref{Ex-1}).
Then
${\rm Ind}(A)=1$
and
$AB\neq BA$.
Therefore,
$A$ and $B$ do not satisfy  $A \mathop \le \limits^{D} B$.
\end{example}

\section{The core-minus partial order }
\label{Section-Core-Minus-Order}

Creating new partial orders is
a  fundamental  problem in matrix theory \cite{Mitra1992laa}.
Matrix decomposition is an important tool of establishing   partial orders.
It is well known that
the C-N partial order was created by Mitra and Hartwig,
and it implies the minus partial order \cite{Mitra1992laa,Mitra2010book}.
  The  minus partial order ``$\ \mathop \le \limits^ - \ $'' ,
  the sharp partial order  ``$\ \mathop \le\limits^\#  \ $''
  and
  the C-N partial order ``$\ \mathop \le \limits^ {\# , - }\ $''
are defined as follows {\rm \cite{Hartwig1980mj,Mitra1987laa,Mitra2010book}}:
\begin{itemize}
  \item[{\rm(i)}]
$A \mathop \le \limits^ - B$
:
$A,B\in \mathbb{C}_{m, n}$,
${\rm rk}(B)-{\rm rk}(A)={\rm rk}(B-A)$;

  \item[{\rm(ii)}]
   $A  \mathop \le \limits^  {\#} B$
   :
   $A,B \in \mathbb{C}_n^{\mbox{\rm\footnotesize\texttt{CM}}}$,
$A^\# A = A^\# B$
and
$ AA^\# = BA^\# $;

  \item[{\rm(iii)}]
$A  \mathop \le \limits^ {\# , - }B$
:
$A,B\in \mathbb{C}_{n, n}$,
 $A_1 \mathop \le \limits^\# B_1 $
and $A_2 \mathop \le \limits^ -  B_2 $,
in which  $A = A_1 + A_2 $ and $B = B_1 + B_2 $
are the core-nilpotent decompositions of $A$ and $B$, respectively. 
\end{itemize}

By using the method
similar to  the C-N partial ordering
as in (iii) ,
we introduce the core-minus partial order in this section.

\begin{definition}
\label{Definition-5-1}
Let $A$ and $B$ be matrices of the same order. Let
$
A = A_1 + A_2 $ and $ B = B_1 + B_2 ,
$
be the core-EP decompositions of $A$ and $B$ respectively,
 where
 $ A_1$ and $B_1$ are   core-invertible,
 $A_2$ and $B_2$ are nilpotent.
Then $A$ is below $B$  under
the core-minus order if
\begin{align}
\label{5-1}
A_1 \mathop \le \limits^{\tiny\textcircled{\#}} B_1 ,
\ \
A_2 \mathop \le \limits^ - B_2 ,
 \end{align}
 Whenever this happens,
  we write
  $A\mathop \le \limits^{\mbox{\rm\footnotesize\texttt{CM}}} B$.
\end{definition}

Since the core-EP decomposition
of a given matrix is unique,
and the core order and the minus order
are both partial orders,
it is easy to prove the following theorem:
\begin{theorem}
\label{Theorem-5-0}
The core-minus order  is a partial order.
 \end{theorem}

\begin{remark}
\label{Remark-5-1}
 When $k\leq 1$,
it is easy to check that
each of the core-minus partial
 coincide with
the core partial order\cite[Definition 2]{Baksalary2010lma}.
\end{remark}

\begin{theorem}
\label{Theorem-5-1}
The core-minus partial order  implies the  minus partial order
i.e.,
if
$A\mathop \le \limits^{\mbox{\rm\footnotesize\texttt{CM}}} B$,
then
$A\mathop \le \limits^{-} B$.
 \end{theorem}

\begin{theorem}
\label{Theorem-5-2}
Let $A$ and $B$ be square matrices of the same order over the complex field.
Then $A\mathop \le \limits^{\mbox{\rm\footnotesize\texttt{CM}}}B$
if and only if
there exists a  unitary matrix $U$
such that
\begin{align}
\label{5-2}
A = U\left[ {{\begin{matrix}
 {T_1 }   & {T_2 }   & {S_1 }   \\
 0   & 0   & 0   \\
 0   & 0   & {N_1 }   \\
\end{matrix} }} \right]U^\ast  ,\
B = U\left[ {{\begin{matrix}
 {T_1 }   & {T_2 }   & {S_1 }   \\
 0   & {T_3 }   & {S_2 }   \\
 0   & 0   & {N_2 }   \\
\end{matrix} }} \right]U^\ast  ,
 \end{align}
where  $T_1 $ and $T_3  $ are non-singular,
 $N_1 $ and $N_2 $ are nilpotent satisfying
 $N_1 \mathop \le \limits^ - N_2  $.
 \end{theorem}

\begin{proof}
Let $A$ and $B$  have the  form as in (\ref{5-2}),
then
we have
the core-EP decompositions of $A$ and $B$,
i.e.
$A = A_1 + A_2$ and $B = B_1 + B_2 $, respectively,
\begin{align*}
 A_1 & = U\left[ {{\begin{matrix}
 {T_1 }   & {T_2 }   & {S_1 }   \\
 0   & 0   & 0   \\
 0   & 0   & 0   \\
\end{matrix} }} \right]U^\ast  ,
\
A_2 = U\left[{{\begin{matrix}
 0   & 0   & 0   \\
 0   & 0   & 0   \\
 0   & 0   & {N_1 }   \\
\end{matrix} }} \right]U^\ast  , \\
 B_1 & = U\left[{{\begin{matrix}
 {T_1 }   & {T_2 }   & {S_1 }   \\
 0   & {T_3 }   & {S_2 }   \\
 0   & 0   & 0   \\
\end{matrix} }} \right]U^\ast  ,
\
B_2 = U\left[{{\begin{matrix}
 0   & 0   & 0   \\
 0   & 0   & 0   \\
 0   & 0   & {N_2 }   \\
\end{matrix} }} \right]U^\ast  .
 \end{align*}
It follows from Definition \ref{Definition-5-1}
that
 $A\mathop \le \limits^{\mbox{\rm\footnotesize\texttt{CM}}}B$.

Conversely, let  $B = B_1 + B_2 $
be the core-nilpotent
decomposition of $B$.
Since
$A_1 \mathop \le \limits^{\tiny\textcircled{\#}} B_1$,
by applying Lemma \ref{lemma-Bak-2},
we know that there exists a  unitary matrix $U$  such that
\begin{align*}
A_1 = U\left[ {{\begin{matrix}
 {T_1 }   & {T_2 }   & {S_1 }   \\
 0   & 0   & 0   \\
 0   & 0   & 0   \\
\end{matrix} }} \right]U^\ast,
\
B_1 = U\left[{{\begin{matrix}
 {T_1 }   & {T_2 }   & {S_1 }   \\
 0   & {T_3 }   & {S_2 }   \\
 0   & 0   & 0   \\
\end{matrix} }} \right]U^\ast .
 \end{align*}
It follows that 
\begin{align*}
B_2 = U\left[{{\begin{matrix}
 0   & 0   & 0   \\
 0   & 0   & 0   \\
 0   & 0   & {N_2 }   \\
\end{matrix} }} \right]U^\ast,
 \end{align*}
where $N_2 $ is nilpotent.
For $A_2 \mathop \le \limits^ - B_2 $,
we have
\begin{align*}
A_2 = U\left[{{\begin{matrix}
 0   & 0   & 0   \\
 0   & 0   & 0   \\
 0   & 0   & {N_1 }   \\
\end{matrix} }} \right]U^\ast  .
 \end{align*}
where $N_1 $ is nilpotent and $N_1 \mathop \le \limits^ - N_2 $.
\end{proof}

\begin{theorem}
Let $A$ and $B$ be square matrices of the same order over the complex field.
Then
$A\mathop \le \limits^{\mbox{\rm\footnotesize\texttt{CM}}}B$
if and only if
$A\mathop \le \limits^{\tiny\textcircled{\dag}} B$
and
$A-AA^{\tiny\textcircled{\dag}} A\mathop \le \limits^ - B-BB^{\tiny\textcircled{\dag}} B.$
\end{theorem}

\begin{proof}
By applying Theorem \ref{Theorem-4-1}
and Corollary \ref{Cor-3-1}, 
we have the   equivalent characterization of the core-minus partial order.
\end{proof}

\section*{ Acknowledgements:}
Hongxing Wang  was supported partially by
the National  Natural Science Foundation of China (No.11401243),
and the Anhui Provincial Natural Science Foundation (No.1508085QA15).

\end{document}